\documentclass[12pt,a4paper]{amsart}

\usepackage[latin1]{inputenc}
\usepackage{amscd,latexsym,mathrsfs,amsfonts,amssymb,amsmath,amsthm,url}
\usepackage{color}
\usepackage{enumerate,verbatim}
\newtheorem{theorem}{Theorem}[section]

\newtheorem{lemma}[theorem]{Lemma}
\newtheorem{proposition}[theorem]{Proposition}
\newtheorem{definition}[theorem]{Definition}
\theoremstyle{remark}
\newtheorem*{remark}{Remark}
\newtheorem*{remarks}{Remarks}

\DeclareMathOperator{\re}{Re }

\numberwithin{equation}{section}

\newtheorem*{claim}{Claim}

\newcommand{\T}{\mathbb{T}}
\newcommand{\D}{\mathbb{D}}
\newcommand{\N}{\mathbb{N}}
\DeclareMathOperator*{\dist}{dist}

\begin{document}
\title[Zero-free approximation and universal o.p.a.]{Simultaneous zero-free approximation \\ and universal optimal polynomial approximants}
\author[B\'en\'eteau]{Catherine B\'en\'eteau}
\address{Department of Mathematics, University of South Florida, 4202 E.~Fowler Avenue,
Tampa FL 33620-5700, USA.} \email{cbenetea@usf.edu}
\author[Ivrii]{Oleg Ivrii}
\address{Department of Mathematics, California Institute of Technology,
1200 E.~California Boulevard, Pasadena CA 91125, USA.} \email{ivrii@caltech.edu}
\author[Manolaki]{Myrto Manolaki}
\address{School of Mathematics and Statistics, University College Dublin, Belfield, Dublin 4, Ireland.} \email{arhimidis8@yahoo.gr}
\author[Seco]{Daniel Seco}
\address{Universidad Carlos III de Madrid and Instituto de Ciencias  
Matem\'aticas (ICMAT), Departamento de Matem\'aticas UC3M, Avenida de  
la Universidad 30,  28911 Legan\'es (Madrid),
Spain.}
\email{dsf$\underline{\,\,\,}$cm@yahoo.es}
\date{\today}

\begin{abstract}
Let $E$ be a closed subset of the unit circle of measure zero. Recently, Beise and M\"uller showed the existence of a function in the Hardy space $H^2$  for which the partial sums of its Taylor series approximate any continuous function on $E$.
In this paper, we establish an analogue of this result in a non-linear setting where we consider \emph{optimal polynomial approximants} of reciprocals of functions in $H^2$ instead of Taylor polynomials. The proof uses a new result on simultaneous zero-free approximation of independent interest. Our results extend to Dirichlet-type spaces $\mathcal{D}_\alpha$ for $\alpha \in [0,1]$.
\end{abstract}

\maketitle

\section{Introduction}\label{Intro}

Given a Hilbert space $\mathscr H$ of analytic functions on the unit disc $\mathbb{D}$ and a non-zero function $f \in \mathscr H $,
 a polynomial $Q_n$ is called an \emph{optimal polynomial approximant (o.p.a.)} of order $n$ of $1/f$ if $Q_n$ minimizes $\|qf - 1\|_{\mathscr H}$ over all polynomials $q$ of degree at most $n$. We note that this definition makes sense even if $f \in \mathscr H$ has zeros in $\mathbb{D}$. In fact, unless $f$ is identically $0$, the polynomial $Q_n$ is uniquely determined.

Optimal polynomial approximants were introduced in \cite{opa0} and are closely connected to classical objects in function theory such as  orthogonal polynomials and reproducing kernels, see \cite{opa2, opa3}.
They arise naturally in connection with the notion of cyclicity in Dirichlet-type
spaces $\mathcal D_{\alpha}$. Such spaces consist of all holomorphic functions $f : \mathbb{D}\to\mathbb{C}$ whose Taylor coefficients in the expansion
$$f(z) =\sum _{k=0}^{\infty} a_n z^n, \qquad z \in \mathbb{D},$$
satisfy
$$ \| f \|_{\alpha}^2:=\sum _{k=0}^{\infty} (k+1)^{\alpha}|a_k|^{2}<\infty.$$  Given also $g(z) = \sum_{k=0}^{\infty}b_kz^k$ in $\mathcal D_{\alpha}$, with the associated inner product
\begin{equation*}
     \langle f , g \rangle_\alpha = \sum_{k=0}^\infty (k+1)^{\alpha} a_k \overline{b_k},
\end{equation*}
the space $\mathcal D_{\alpha}$ becomes a reproducing kernel Hilbert space. Of particular interest are the cases of Bergman, Hardy and Dirichlet spaces, corresponding respectively
 to the parameters $\alpha=-1,0,1$. We refer the reader to \cite{duren-schuster, duren, EFKMR} for background on these spaces. In our paper, we focus on the case of the Hardy space $H^2$ and the Dirichlet space $\mathcal{D}$.

A function $f \in \mathcal D_{\alpha}$ is called \emph{cyclic} if the set of polynomial multiples of $f$ is dense in $\mathcal D_\alpha$.  By a result of Beurling, a function in $H^2$ is cyclic if and only if it is outer, see \cite{duren}. On the other hand, cyclic functions for the classical Dirichlet space $\mathcal{D}$ have not been fully understood: In \cite{brown-shields}, Brown and Shields provided necessary conditions and conjectured these to be sufficient.  Partial advances on this problem are described in \cite{EFKMR}. The problem in the Bergman space $A^2$ is also far from solved.
However, it is easy to provide a characterization of cyclicity using the notion of optimal polynomial approximants: a function $f \in \mathcal D_{\alpha}$  is cyclic if and only if the optimal polynomial approximants $Q_n$ of
$1/f$ satisfy $\| Q_nf -1\| _{\alpha} \to 0$ as $n\to \infty$. We note that the latter condition implies that $Q_n$ converges to $1/f$ locally uniformly on $\mathbb{D}$.  It is natural to ask whether this convergence still holds on the boundary.  In \cite{opa1}, it is observed that if a polynomial $f$ has only simple roots, all of which lie outside the unit disc, then $Q_n$ converges uniformly to $1/f$ on any closed subset  of the unit circle which does not intersect the roots of $f$.

In this paper, we are interested in the opposite phenomenon: are there functions with o.p.a.~that have more than one limit point at some $\zeta$ on the unit circle $\T$?  Can those limit points be the entire complex plane $\mathbb{C}$? This question relates to the concept of universality.  Roughly speaking, a function is called {\em universal}\/ if using a sequence of mappings, it can approximate every object in some class of interest.

 In \cite{nestoridis}, Nestoridis   showed that there exists a holomorphic function $f$ in $\D$ such that for every compact set $K \subset \mathbb{C} \setminus \mathbb{D}$ with connected complement and every  continuous function $g:K\to \mathbb{C}$ which is holomorphic in the interior of $K$, there exists a subsequence of Taylor polynomials of $f$ which converges uniformly to $g$ on $K$.
  Even though holomorphic functions with universal Taylor series are generic in the sense of Baire, no explicit example is known.
  It turns out that all such functions have wild boundary behaviour, see
\cite{gardiner14,gardiner-khavinson,gardiner-manolaki3,gardiner18}.
 This erratic behaviour occurs even if we merely have universal approximation on a large subset of the unit circle: for example,  in \cite{gardiner-manolaki2}, the authors showed that if $E \subset \mathbb{T}$ is a closed set of positive measure and $f$ has a universal Taylor series on $E$, then the set
$f(T_{\zeta})$ is dense in $\mathbb{C}$ for almost every $\zeta\in E$, where $T_{\zeta}\subset \mathbb{D}$ is a Stolz angle with vertex at $\zeta$.
 On the other hand, Beise and M\"uller \cite{beise-muller} showed that if $E \subset \mathbb{T}$ is a closed set of measure $0$, then there exists a function with
a universal Taylor series on $E$ which lies in the Hardy space $H^p$, $1\le p<\infty$ (and therefore has non-tangential limits almost everywhere on $\mathbb{T}$).

The main goal of this paper is to establish an analogue of the result of Beise and M\"uller for o.p.a.~of reciprocals of functions in $H^2$.   To formulate our main result, we introduce some notation.  For
$f \in (H^2)^{*} := H^2 \setminus \{0\}$, we denote by $\bigl (Q_n(1/f) \bigr)$ the sequence of o.p.a.~of the reciprocal of $f$.  For a closed set $E \subset \T,$ we denote by $C(E)$ the space of continuous functions on $E$ equipped with the topology of uniform convergence on $E$.

\begin{definition}\label{universalE}
We say  that a function  $f \in (H^2)^{*}$ belongs to the class $\mathcal{U}_{E}$ of functions with universal o.p.a.~ on $E$ if for each $g\in C(E),$ there exists a subsequence  of o.p.a.~$\bigl (Q_{m_s}(1/f) \bigr )$ such that $Q_{m_s}(1/f)\to g$ in $C(E).$
\end{definition}

We show the following result:

\begin{theorem}\label{universal opa}
If $E \subset \T$ is a closed set of measure $0$, then $\mathcal{U}_{E}$ is a $G_\delta$ dense subset of $H^2$.
\end{theorem}

\begin{remark}
In particular, Theorem \ref{universal opa} implies that for any $\zeta \in \T,$ there are functions $f \in H^2$ for which $\left\{ Q_n(1/f)(\zeta): n \in \N \right\}$ is dense in $\mathbb{C}.$
\end{remark}

One crucial ingredient in the case of universal Taylor series is a result on simultaneous polynomial approximation in $H^2 \times C(E)$, see \cite[Lemma 2.1]{beise-muller}. For the purposes of our paper,  we needed to prove a similar result involving simultaneous  zero-free  approximation, which is of independent interest.

\begin{theorem}
\label{simul-approx}
Let $E \subset \T$ be a closed set of measure $0.$  Let $g \in H^2$ be non-vanishing on $\D$, and let $f \in C(E)$ be non-vanishing on $E$. Then there exists a sequence of polynomials $P_n$ which do not vanish on the closed unit disc $\overline{\mathbb{D}}$ such that $P_n \to g$ in $H^2$ and $P_n \to f$ in $C(E)$.
\end{theorem}

The corresponding result in \cite{beise-muller} uses functional analysis techniques which cannot be adapted to our zero-free setting. Instead, we give a  constructive  argument based  on  a  modification  of  the  Rudin-Carleson  Theorem \cite{rudin, carleson}. In many situations, polynomial  approximation in  the  non-vanishing setting is a  delicate issue. For example, it is not known if there is a non-vanishing version of the classical Mergelyan's Theorem; see \cite{andersson, andersson-gauthier, danielyan, gauthier-knese} for partial results and connections with the Riemann Hypothesis. 

We first show the existence of universal optimal polynomial approximants for the Hardy space in Section \ref{sec:existence-opa}, subject to the verification of the result on simultaneous zero-free approximation, which we prove in Section \ref{sec:simul-approx}. Then, in Section \ref{sec4}, we establish the corresponding results for the case of the Dirichlet space $\mathcal D$ (see Theorems \ref{universal dirichlet} and \ref{simul-approxD}), in line with the recent work of M\"uller \cite{muller2}. Finally, in Section \ref{conclusion}, we conclude with some remarks and possible directions for future research.

\section{Existence of functions with universal optimal polynomial approximants}
\label{sec:existence-opa}
In this section, we prove Theorem \ref{universal opa} on the existence of functions in $H^2$ whose optimal polynomial approximants are universal on a given closed subset of the unit circle of measure zero. We rely on three key properties of $H^2$: the action of the shift is isometric, there exist almost everywhere well-defined boundary values, and functions factor as a product of their inner and outer parts.  We begin with some auxiliary results:
\begin{proposition}\label{innercontinuity}
Let $n \in \N$ and let $K \subset \mathbb{C}$ be compact. The mapping $\mathcal{Q}_n: (H^2)^* \rightarrow C(K)$ which takes a function $f$ to $Q_n(1/f)$ is continuous.
	\end{proposition}
\begin{proof}
Let $n\in\mathbb{N}$, and let $f \in (H^2)^*.$
We denote by $a_j$ the $j^{th}$ coefficient of the optimal polynomial approximant $Q_n(1/f)$.
Let $\mathcal{P}_n$ the space of polynomials of degree at most $n$.
Since $Q_n(1/f)f$ is the orthogonal projection of $1$ onto the space $\mathcal{P}_nf$,
$$\bigl ( Q_n(1/f)f-1 \bigr ) \perp z^t f, \ \ \ t=0, \dots, n.$$
It follows that the vector $A=(a_j)_{j=0}^n$ of the Taylor coefficients of $Q_n(1/f)$ is given as the unique solution to the linear system
$$AM=C,$$
where $M$ is the $(n+1)\times (n+1)$ matrix with entries $$M_{j,k}=\langle z^j f,z^k f \rangle \ \ \ j,k=0,\dots,n$$ and $C$ is the vector given by $$C=(c_j)_{j=0}^n= \bigl (\langle 1,z^j f \rangle \bigr )_{j=0}^n= \bigl (\overline{f(0)}, 0, \dots, 0 \bigr).$$
Since the vector $C$ and matrix $M$ vary continuously in $f \in (H^2)^{*}$, and the matrix $M$ is invertible, the vector $A$ also varies continuously in $f \in (H^2)^{*}$.
The continuity of the mapping $\mathcal{Q}_n$ follows.
\end{proof}

As noted in \cite{opa4}, the optimal polynomial approximants in $H^2$ are essentially determined by their outer part.  The following proposition makes this dependence explicit.

\begin{proposition}\label{invinner}
If $g$ is an inner function in $H^2$ and $f \in (H^2)^*,$  then, for each $n \in \mathbb{N}$,
 $$Q_n \bigl (1/(fg) \bigr ) =\overline{g(0)} \, Q_n(1/f).$$
\end{proposition}
\begin{proof}
For the function $f \in (H^2)^*$, let $A$, $C$ and $M$ be as in the  proof of Proposition \ref{innercontinuity}. Similarly,
the vector $B=(b_j)_{j=0}^n$ of the Taylor coefficients of $Q_n \bigl (1/(fg) \bigr )$ is given as the unique solution to the linear system
$$BN=D,$$
where $N$ is the $(n+1)\times (n+1)$ matrix with entries $$N_{j,k}=\langle z^j fg,z^k fg \rangle\ \ \ j,k=0,\dots,n$$ and $D$ is the vector given by $$D=(d_j)_{j=0}^n= \bigl (\langle 1,z^j fg \rangle \bigr )_{j=0}^n=\bigl (\overline{f(0)}\overline{g(0)}, 0, \dots, 0 \bigr)=\overline{g(0)}C.$$
Since $g$ is an inner function in $H^2$, multiplication by $g$ is an isometry and thus, for all $j$ and $k$,
$$\langle z^j fg,z^k fg \rangle=\langle z^j f,z^k f \rangle,$$
which implies that the matrices $M$ and $N$ are equal. Hence
$$B=DN^{-1}=\overline{g(0)}CM^{-1}=\overline{g(0)}A,$$
which gives the desired conclusion that $Q_n \bigl (1/(fg) \bigr ) =\overline{g(0)} \, Q_n(1/f).$
\end{proof}

\begin{remarks}
(i) The above proposition shows that if $g$ is an inner function with $g(0) \ne 0$, then $f \in (H^2)^*$ has universal o.p.a.~if and only if $fg \in (H^2)^*$ has universal o.p.a.

(ii) Each function in $H^2$ can be written as a product of an inner and an outer function. Since cyclic functions in $H^2$ are precisely the outer functions, if there exists a function with universal o.p.a.~on some set, then there exists a cyclic function with the same property. In particular, Theorem \ref{universal opa} tells us that, for any closed set $E \subset \T$ of zero measure, there exists a cyclic function $f \in H^2$ with universal o.p.a.~on $E$.  In this case, $Q_n(1/f)$ converges to $1/f$ locally uniformly on $\mathbb{D}$ while $\bigl \{Q_n(1/f): n\in\mathbb{N}\bigr \}$ is dense in $C(E)$.

(iii) Let $(z_n)$ be a (finite or infinite) sequence in $\mathbb{D} \setminus \{0\}$ which satisfies the Blaschke condition
 $$\sum _{n=1}^{\infty} (1-|z_n|)<\infty.$$ Then there exists a function $f$ with universal o.p.a.~having zeros at $(z_n)$. Indeed, we can obtain such a function by multiplying a cyclic function with universal o.p.a. (which is zero-free on $\mathbb{D}$) with a suitable Blaschke product.
\end{remarks}

To establish the existence of functions with universal optimal polynomial approximants in  $H^2$,  we will use the Baire category theorem.
We recall that the collection of functions with universal optimal approximants on $E$ is defined as:
$$
\mathcal{U}_{E} = \bigl \{
 f\in (H^2)^{*} : \forall g\in C(E) \, \exists (m_s) : Q_{m_s}(1/f)\to g \mbox{ in } C(E) \bigr \}.
$$
Let $\{P_n : n\in\mathbb{N}\}$ be the collection of all polynomials with rational coordinates that are zero-free on $E$. This set is clearly dense in $C(E)$.
For each $k,n,m \in \mathbb{N}$, we consider the set
$$
E_{k,n,m}= \bigl \{ f\in (H^2)^{*}:  \|Q_{m}(1/f)-P_n\|_{C(E)}<1/k \bigr \}.
$$
In view of the Baire category theorem, in order to complete the proof of Theorem \ref{universal opa}, it suffices to prove the following proposition:

\begin{proposition}\label{keyproposition}
Let $E \subset \T$ be a closed set of measure $0.$ Then:
\begin{itemize}
\item[(a)] The set $\mathcal{U}_{E}$ can be written as
$$\mathcal{U}_{E}=\bigcap _{k,n=1}^{\infty}\bigcup _{m=1}^{\infty} E_{k,n,m}.$$

\item[(b)] For each $k,n,m \in \mathbb{N}$, the set $E_{k,n,m}$ is open in $H^2$.

\item[(c)] For each $k,n \in \mathbb{N}$, the set $\bigcup _{m=1}^{\infty} E_{k,n,m}$ is dense in $H^2$.
\end{itemize}
\end{proposition}

Indeed, the above proposition implies that $\mathcal{U}_{E}$ can be expressed as a countable intersection of open dense sets in $H^2$. Since $H^2$ is  a complete metric space, we can apply the Baire category theorem to deduce that  $\mathcal{U}_{E}$ is a $G_\delta$ dense subset of $H^2$, which establishes Theorem \ref{universal opa}.

\begin{proof}
(a) Let $f\in\mathcal{U}_{E}$ and let $k,n \in \mathbb{N}$. Since $P_n \in C(E)$, there is an $m \in \mathbb{N}$ such that $\|Q_{m}(1/f)-P_n\|_{C(E)}<1/k$. Thus $f\in \bigcup _{m=1}^{\infty} E_{k,n,m}$. Hence
$$
\mathcal{U}_{E}\subset\bigcap _{k,n=1}^{\infty}\bigcup _{m=1}^{\infty} E_{k,n,m}.
$$
To show the reverse inclusion, let $f\in\bigcap _{k,n=1}^{\infty}\bigcup _{m=1}^{\infty} E_{k,n,m}$.
Fix a continuous function $g\in C(E)$ and  $\varepsilon >0$. Since $\{P_n : n\in\mathbb{N}\}$ is dense in $C(E)$, we can choose $P_{n}$ so that $\|P_n-g\|_{C(E)}<\varepsilon/2$.  Choose a large integer $k \ge 1$ for which $1/k<\varepsilon/2$. For these choices of $n$ and $k$, there is an $m\in \mathbb{N}$ such that $\|Q_m(1/f)-P_n\|_{C(E)}<1/k$.
The triangle inequality gives $\|Q_m(1/f) - g\|_{C(E)} < \varepsilon$.

Since $\varepsilon > 0$ can be taken to be arbitrarily small, we can find a subsequence $\bigl (Q_{m_s}(1/f) \bigr)$ of o.p.a.~of $1/f$ such that $Q_{m_s}(1/f)\to g$  uniformly on $E$. This implies that
$$\mathcal{U}_{E}\supset\bigcap _{k,n=1}^{\infty}\bigcup _{m=1}^{\infty} E_{k,n,m},$$
which gives the desired conclusion.

(b) Fix $k,n,m \in \mathbb{N}$. It is easy to see that $E_{k,n,m}$ is the inverse image of the open ball in $C(E)$ with center at $P_n$ and radius $1/k$ via the mapping $\mathcal Q_m$.  By Proposition \ref{innercontinuity}, the mapping
 $\mathcal Q_m: (H^2)^* \to C(E)$ is continuous, and therefore $E_{k,n,m}$ is open in $(H^2)^*$. Since  $(H^2)^*$ is open in $H^2$,
$E_{k,n,m}$  is also open in $H^2$.

(c) Let $k,n \in \mathbb{N}$. For simplicity, we write $g:=P_n$.
To show that $\bigcup _{m=1}^{\infty} E_{k,n,m}$ is dense in $H^2$, it suffices to prove the following claim:
\begin{claim}
 For any  $f \in H^2$ and $\varepsilon > 0$, there exists a function $F \in (H^2)^*$ and an integer $m \ge 1$ such that
\begin{equation}
\label{eq:desired-properties}
 \|F - f\|_{H^2} < \varepsilon, \qquad \|Q_m(1/F)-g\|_{C(E)} < \varepsilon.
\end{equation}
\end{claim}

 Without loss of generality, we may assume that $f$ extends holomorphically to some  neighbourhood of the closed unit disc and that $f(0) \ne 0$, as this class of functions is dense in $H^2$.
Since $f\in H^2$, we can write $$f =  f_I \cdot f_O, $$ where $f_I$ is an inner function and $f_O$ is an outer function.
Since $f(0) \neq 0$, we have $f_I(0) \neq 0$ and therefore, the function $h:=(\overline{f_{I}(0)})^{-1} f_{O}\in H
^2$ has no zeros in the unit disc. Meanwhile, $1/g\in C(E)$ has no zeros on $E$.
 Thus, we can apply Theorem \ref{simul-approx} (which we will prove in Section \ref{sec:simul-approx}) to find a polynomial $P$ which does not vanish in the closed unit disc and satisfies
\begin{equation}
\label{P-h}
 \|P-h\|_{H^2} < \varepsilon
 \end{equation}
 and
\begin{equation*}
 \|P-1/g\|_{C(E)} < \delta ,
 \end{equation*}
where $\delta = \delta(g) > 0$ is chosen sufficiently small to guarantee
\begin{equation}\label{e}
\|g-1/	P\|_{C(E)} < \varepsilon /2.
\end{equation}

We will show that the function $F:=\overline{f_{I}(0)} f_{I}\cdot P$ satisfies the required inequalities (\ref{eq:desired-properties}). From Proposition \ref{invinner}, we know that $Q_m(1/P) = Q_m(1/F)$ for any $m \ge 1$.
Since $P$ is a polynomial with no zeros in the closed unit disc, we know (see, e.g., \cite{opa1}) that $Q_m(1/P)\to 1/P$ uniformly on the unit circle $\mathbb{T}$, and in particular on $E$.
 Thus, by choosing $m$ sufficiently large,  we can make
$$\|Q_m(1/F)-1/P\|_{C(E)}= \|Q_m(1/P)-1/P\|_{C(E)}< \varepsilon/2,
$$
which, combined with (\ref{e}) and the triangle inequality, implies that
$$\|Q_m(1/F)-g\|_{C(E)}<\varepsilon.$$
Finally,  using that $f_I$ is inner and (\ref{P-h}), we deduce $$\|F-f\|_{H^2} =\|\overline{f_{I}(0)} f_{I}(P - h)\|_{H^2} \le \|\overline{f_{I}(0)} f_{I}\|_{H^\infty}\cdot\|P - h\|_{H^2}  < \varepsilon,$$
  which establishes the desired claim.
\end{proof}

Thus the proof of Theorem \ref{universal opa} is complete, subject to the verification of Theorem \ref{simul-approx}.

\section{Simultaneous zero-free approximation}
\label{sec:simul-approx}

Let us now turn to the proof of Theorem \ref{simul-approx}. Our argument is motivated by a classical theorem of Rudin \cite{rudin} and Carleson \cite{carleson}, which says that {\em if $E$ is a closed subset  of measure zero on the unit circle and $f$ is a continuous function on $E$, then there exists a function $F$ analytic on the unit disc which extends continuously to the unit circle and agrees with $f$ on $E$.}

\begin{proof}[Proof of Theorem \ref{simul-approx}]

Let  $E \subset \mathbb{T}$ be a closed subset of measure zero,
 $g \in H^2$,  and $f\in C(E)$. Assume that $g$ does not vanish on the  unit disc and $f$ does not vanish on $E$. We want to construct a sequence of polynomials $P_n$ which do not vanish on the closed  unit disc such that $P_n \to g$ in $H^2$ and $P_n \to f$ in $C(E)$.
To simplify the construction, we employ the following reductions:

(i)
We may assume that $g$ is continuous up to the unit circle and is non-vanishing on the closed unit disc. Indeed, if for any $0 < r < 1$ we can find a sequence of (zero-free on $\overline{\mathbb{D}}$) polynomials which approximates $(g(rz), f(z))$, a diagonalization argument would produce a sequence of (zero-free on $\overline{\mathbb{D}}$) polynomials which approximates $(g(z), f(z))$.

(ii)
Instead of approximating $(g(z), f(z))$ by polynomials, we may approximate by functions $g_n$ in $A(\mathbb{D})$, that is, functions holomorphic on $\mathbb{D}$ and continuous on $\overline{\mathbb{D}}$. We can obtain an approximating sequence of
polynomials by taking Taylor polynomials of $g_n(r_nz)$ for suitable $r_n\in (0, 1)$.

(iii) It suffices to treat the case when  $f/g$ is piecewise-constant, that is, when $E = \bigcup_{j=1}^m E_j$ is a union of finitely many disjoint closed sets and $f/g = e^{v_j}$ on $E_j$ for some $v_j \in \mathbb{C}$. Indeed, since $f/g$ is continuous on $E$, for any $\varepsilon > 0$, we can choose $k \ge 1$ sufficiently large so that
$$
|(f/g)(\zeta_1) - (f/g)(\zeta_2)| < \varepsilon, \quad \zeta_1, \zeta_2 \in E, \quad |\zeta_1 - \zeta_2| < 4\pi/k.$$
We divide the circle into $k$ arcs of equal length.  Since $E$ is nowhere dense, we may slightly modify the arcs so that their endpoints are not contained in $E$ and their lengths are at most $4\pi/k$. Call the resulting arcs $I_1, I_2, \dots, I_k$.
Set $E_j := I_j \cap E$, discard the empty pieces and renumber the remaining sets from $1$ to $m$.
For each $j=1,\dots, m$, let $\zeta_j$ be an arbitrary point of $E_j$. Since $(f/g)(\zeta _j)\neq 0$ (by reduction (i)), there exists a complex number $v_j$ such that $e^{v_j}=(f/g)(\zeta _j)$. It follows that $|f/g - e^{v_j}| < \varepsilon$ on $E_j$ for any $j = 1, 2, \dots, m$. This construction shows that the space of
 pairs $(g,f)$ with $f/g$ piecewise-constant is dense in the space of all pairs covered by the theorem.

With the above reductions, to prove Theorem \ref{simul-approx}, we  construct a sequence of {\em uniformly bounded} functions $\Phi_n \in A(\mathbb{D})$, which are zero-free on $\overline{\mathbb{D}}$, such that $\Phi_n \to 1$ uniformly on compact subsets
 of $\overline{\mathbb{D}} \setminus E$ and $\Phi_n \to e^{v_j}$ uniformly on each $E_j$. Since $E$ is a closed set of measure $0$, the sequence $g_n = g \cdot \Phi_n$ tends to $g$ in $H^2$,
  while $g_n$ tends uniformly to $f$ on $E$.
To this end, we take
\begin{equation}
\label{eq:phin-def}
\Phi_n(z) \, = \,   \exp \biggl (\sum_{j=1}^m v_j \cdot h_{j,n}(z) \biggr ),
\end{equation}
where $h_{j,n}$ are the ``Rudin functions'' given by the lemma below, which is inspired by Lemma 1 of \cite{rudin}:

\begin{lemma}\label{fundlemma}
Let $E \subset \mathbb{T}$ be a closed set of measure 0 and $U$ be an open neighbourhood of $E$.
 For any $\varepsilon > 0$, there exists a function $h \in A(\mathbb{D})$ which satisfies
$$
|h(z)| \le 2, \ z \in \overline{\mathbb{D}}, \qquad
|h(z)| < \varepsilon, \ z \in \overline{\mathbb{D}} \setminus U,$$
$$
h(z) =1, \ z \in E.
$$
\end{lemma}

\begin{proof}[Proof of Lemma \ref{fundlemma}]
 Let $\varepsilon >0$. Also let $V \supset E$ be an open set compactly contained in $U$ and $\delta >0$.
Since $E$ has measure 0, we can construct a real-valued function $u(\zeta) \ge 0$  in $C^\infty(\mathbb{T} \setminus E)$
which tends to $+\infty$ as $\zeta \to w$ for every $w \in E$,  is identically 0 on $\mathbb{T} \setminus \overline{V}$ and satisfies
\begin{equation}
\label{eq:l1norm-u}
0  \, < \, u(0) \, = \, \frac{1}{2\pi} \int_{\mathbb{T}} u(\zeta) |d\zeta| \, < \, \delta,
\end{equation}
where we continue to use the symbol $u$ for the Poisson extension to the unit disc. Let $\tilde u$ be the harmonic conjugate of $u$, normalized so that $\tilde u(0) = 0$.
 Note that the harmonic conjugate $\tilde u \in C^\infty(\overline{\mathbb{D}} \setminus E)$ since the Hilbert transform preserves smoothness, see \cite{garnett, garnett-marshall}.

We now consider the function $h_1 = \exp(-u-i\tilde u).$
First of all, since $u, \tilde u \in C^\infty(\overline{\mathbb{D}} \setminus E)$, the function $h_1$ is holomorphic in $\D$ and extends continuously to $\mathbb{T} \setminus E$. However, on $E$, $h_1$ extends continuously to $0$ (since $u$ tends to infinity).
Thus, $h_1$ extends continuously to the unit circle.
Since $u$ is non-negative, $|h_1| \le 1$.
Finally, from the Herglotz representation
$$
u(z) + i\tilde u(z) = \frac{1}{2\pi} \int_{\mathbb{T}} u(\zeta) \cdot \frac{\zeta+z}{\zeta-z} \, |d\zeta|,
$$
it follows that   $u + i \tilde u $ is very small on $\overline{\mathbb{D}} \setminus U$:
\begin{align*}
|u(z)+i\tilde u(z)| & \le \frac{1}{2\pi} \int_{\overline{V}\cap \T} u(\zeta) \cdot \biggl | \frac{\zeta+z}{\zeta-z} \biggr | \, |d\zeta| \\
& \le
\frac{1}{\pi \cdot \dist(\overline{\mathbb{D}} \setminus U, \, \overline{V} \cap \T)} \int_{\overline{V}\cap \T} u(\zeta) |d\zeta| \\
& \le
\frac{2\delta}{ \dist(\overline{\mathbb{D} }\setminus U, \, \overline{V} \cap \T)}.
\end{align*}
  Thus, by requiring $\delta > 0$ in (\ref{eq:l1norm-u}) to be sufficiently small, we can make  $|h_1 - 1| < \varepsilon$ on $\overline{\mathbb{D}} \setminus U$.

Since $h := 1 - h_1$ has the required properties, the proof of Lemma \ref{fundlemma} is complete.
\end{proof}

To finish the proof of the theorem, we apply Lemma \ref{fundlemma} with the closed sets $E_j$,  open neighbourhoods $U_{j,n}=\{z \in \mathbb{C}: \dist(z, E_j) < 1/n\}$ and $\varepsilon = 1/n$.  We call $h_{j,n}$ the resulting Rudin functions and define $\Phi_n$ by the formula (\ref{eq:phin-def}). Note that when $n \ge 1$ is large, the open neighbourhoods $U_{j,n}$ are disjoint.  It is easy to see that the  functions $\Phi_n$  are uniformly bounded by $\exp \bigl (2m \cdot \max \{|v_j|:j=1,\dots,m\} \bigr )$ on $\overline{\mathbb{D}}$ and have all the desired properties. Thus, the proof of Theorem \ref{simul-approx} is complete.
\end{proof}

\section{Universality in Dirichlet-type spaces}\label{sec4}
Recently, M\"uller \cite{muller2} showed that for each closed set $E \subset \T$ of (logarithmic) capacity zero, there is a function in the Dirichlet space $\mathcal D$ that has a universal Taylor series on $E$. It is natural to consider the analogous question for the case of optimal polynomial approximants.

We note that the Dirichlet space $\mathcal D$ consists of all holomorphic functions $f$ on $\mathbb{D}$ satisfying
$$\mathcal{D}(f):=\dfrac{1}{\pi}\int _{\mathbb{D}} |f'|^2 dA <\infty ,$$
where $A$ denotes the area measure in $\mathbb{C}$. More precisely, for each function $f$ in $\mathcal D$  we have $\|f \| _{\mathcal{D}}^2= \mathcal{D}(f)+ \|f\|_{H^2}^2.$
Let $\mathcal D_{nv} \subset \mathcal D$ denote the (metric)  subspace of $\mathcal D$ formed by all the nowhere vanishing functions and the identically zero function. By Hurwitz theorem, it is a closed subset of the Dirichlet space (and hence it is a complete metric space, so that the Baire category theorem applies). For a subset $E$ of $\T$, we also denote $$
\mathcal{U}^{\mathcal{D}_{nv}}_{E} = \bigl \{
 f\in \mathcal D_{nv}\setminus \{0\}: \forall g\in C(E) \, \exists (m_s) : Q_{m_s}(1/f)\to g \mbox{ in } C(E) \bigr \}.
 \color{black}
$$

In this section, we prove the following theorem:

\begin{theorem}\label{universal dirichlet}
If $E \subset \mathbb{T}$ is a closed set of capacity $0$, then $\mathcal{U}^{\mathcal D_{nv}}_{E}$ is a $G_\delta$ dense subset of $\mathcal D_{nv}$.
\end{theorem}

\begin{remark}
Notice that for $f\in \mathcal{D}\subset H^2$, the optimal polynomial approximants to $1/f$ in $\mathcal{D}$ are not the same polynomials as those in $H^2$ and so the result does not follow from Theorem \ref{universal opa}. In particular, we do not know whether $\mathcal{U}^{\mathcal{D}_{nv}}_{E} \subset \mathcal{U}_{E}$.
\end{remark}

In order to deduce Theorem \ref{universal dirichlet} from the Baire category theorem, we need an analogue of Proposition \ref{keyproposition} for $\mathcal D$, under the stronger assumption
that $E\subset \mathbb{T}$ is a closed set of capacity $0$. To prove analogues of parts (a) and (b), we can argue as in the Hardy case. For part (c), it suffices to establish a simultaneous zero-free approximation result for $\mathcal D \times C(E)$:

\begin{theorem}
\label{simul-approxD}
Let $E \subset \T$ be a closed set of capacity $0.$  Let $g \in \mathcal{D}$ be non-vanishing on $\D$, and let $f \in C(E)$ be non-vanishing on $E$. Then there exists a sequence of polynomials $P_n$ which do not vanish on the closed unit disc $\overline{\mathbb{D}}$ such that $P_n \to g$ in $\mathcal{D}$ and $P_n \to f$ in $C(E)$.
\end{theorem}

For this purpose, we construct an analogue of the Rudin functions tailored for the Dirichlet space:

\begin{lemma}
\label{dirichlet-rudin}
Let $E \subset \mathbb{T}$ be a closed set of capacity $0$ and $U$ be an open neighbourhood of $E$.
 For any $\varepsilon > 0$, there exists a function $h \in A(\mathbb{D})$ which satisfies
$$
|h(z)| \le 2, \ z \in \overline{\mathbb{D}}, \qquad
|h(z)| < \varepsilon, \ z \in \overline{\mathbb{D}} \setminus U,$$
$$
h(z) =1, \ z \in E, \qquad \int_{\mathbb{D}} |h'|^2 dA \le \varepsilon.
$$
\end{lemma}

To prove the above lemma, we adjust the construction in Theorem 3.4.1 of \cite{EFKMR}:
\begin{lemma}
\label{primer-thm}
Let $E$ be a closed subset of $\mathbb{T}$ of capacity $0$ and $U$ be an open neighbourhood of $E$. There exists
$f \in \mathcal D$ with $\re f(z) \ge 0$ such that $\lim_{z \to \zeta} \re f(z) = +\infty$ for all $\zeta \in E$.
We can choose $f$ to be continuous on $\overline{\mathbb{D}} \setminus E$. Additionally, for any $\delta > 0$, we can choose $f$ so that $\| f \|_{\mathcal D} < \delta$ and
$|f| < \delta$ on $\overline{\mathbb{D}} \setminus U$.
\end{lemma}

\begin{proof}[Proof of Lemma \ref{primer-thm}]
Choose a decreasing sequence of closed neighbourhoods $E_n$ of $E$ in $\mathbb{T}$
such that the sum
\begin{equation}
\label{eq:root-sum}
\sum_{n=1}^\infty c(E_n)^{1/2} < \infty ,
\end{equation}
where $c(\cdot)$ denotes the (logarithmic) capacity.
For each $n \ge 1$, let $\mu_n$ be the equilibrium measure for $E_n$. Following the notation from \cite{EFKMR}, set 
$$f_{\mu_n}(z)=\int_{\mathbb{T}} \log \left( \frac{2}{1-z\overline{\zeta}}\right)  \, d\mu_{n}(\zeta)$$
and
 $$f(z) = \sum_{n = 1}^\infty c(E_n) f_{\mu_n}(z) = \sum_{n = 1}^\infty  \int_{\mathbb{T}} c(E_n)\log \left( \frac{2}{1-z\overline{\zeta}}\right)  \, d\mu_{n}(\zeta).$$

As in the proof of Theorem 3.4.1 of \cite{EFKMR}, the sum defining $f$ converges uniformly on compact subsets of the unit disc and $\| f \|_{\mathcal D}$ is finite.

 Let $V$ be an open set compactly contained in $U$. In order for $f$ to satisfy the additional properties in the statement of the lemma, we simply choose the closed neighbourhoods $E_n \subset V$ so that $\sum_{n=1}^\infty c(E_n)^{1/2} $ is as small as we wish.
The function $f$ is small on $\mathbb{D} \setminus U$ since the $\mu_n$ are probability measures: 
$$
|f(z)| = \sum _{n=1}^{\infty} c(E_n) |f_{\mu_n}(z)| \le \sum _{n=1} ^{\infty} c(E_n)
 \int_{\mathbb{T}} \biggl | \log  \left( \frac{2}{1-z\overline{\zeta}}\right)  \biggr | d\mu_n(\zeta).
$$
As in \cite{EFKMR}, $\| f \|_{\mathcal D}$ is arbitrarily small because it is controlled by the sum (\ref{eq:root-sum}). Moreover, arguing as in Theorem 3.4.1 of \cite{EFKMR}, we can choose $f$ to be continuous on $\overline{\mathbb{D}} \setminus E$.
\end{proof}

\begin{proof}[Proof of Lemma \ref{dirichlet-rudin}] By Lemma \ref{primer-thm}, the function defined by $h = 1 - e^{-f}$ on $\overline{\mathbb{D}}\setminus E$  and $h=1$ on $E$ belongs to $A(\mathbb{D})$ and satisfies $\| h \|_ {H^\infty} \le 2$.
By construction, the function $h$ will be close to $0$ on $\overline{\mathbb{D}} \setminus U$. Since
$|h'| = |f'| e^{-\re f}$, we conclude that $h \in \mathcal D$ with
$$
\int_{\mathbb{D}} |h'|^2 dA \le \int_{\mathbb{D}} |f'|^2 dA .
$$
The proof is complete.
\end{proof}

We are now ready to prove the result on simultaneous zero-free approximation in $\mathcal D \times C(E)$, where
$E \subset \mathbb{T}$ has zero capacity. 
\begin{proof}[Proof of Theorem \ref{simul-approxD}] We use the same construction as in the Hardy case:
$$
g_n = g \Phi_n = g \prod_{j=1}^m \Psi_{j,n} = g \exp \biggl (\sum_{j=1}^m v_j h_{j,n} \biggr),
$$
 except we construct the functions $h_{j,n}$ using Lemma \ref{dirichlet-rudin}, since we want
$$
\int_{\mathbb{D}} |h_{j,n}'|^2 dA 
$$
to be arbitrarily small. 

By construction, the functions $\Phi _n$ are uniformly bounded on $\overline{\mathbb{D}}$ and $\Phi_n \to 1$ uniformly on compact subsets of $\overline{\mathbb{D}}\setminus E$.
Differentiating and using the fact that $|a+b|^2\leq 2(|a|^2+|b|^2)$ shows that there exists a positive constant $C$, which depends only on $m$, such that
$$
\int_{\mathbb{D}} |g \Phi_{n}'|^2 dA  \leq C   \sum_{i=1}^m \int_{\mathbb{D}} |\Psi'_{i,n}|^2 \biggl ( |g| \prod_{ \substack{j=1\\ j \ne i}}^m |\Psi_{j,n}| \biggr )^2 dA
.$$
Since the right-hand side of the above inequality can become arbitrarily small and $|(g_n-g)'|^2\leq 2 (|g'(\Phi _{n}-1)|^2+|g \Phi_{n}'|^2)$,
we conclude that $\|g_n-g\|_{\mathcal D} \to 0$ as $n\to\infty$.
\end{proof}
\begin{remark}
It seems likely that the same strategy can be implemented for the generalized Dirichlet space $\mathcal D_{\alpha}$ ($0 < \alpha < 1$) if one replaces
logarithmic capacity with the Riesz capacity of parameter $(1-\alpha)$.
\end{remark}

\section{Concluding remarks and further directions}\label{conclusion}

We conclude with some remarks and open problems:

\begin{enumerate}
\item
As discussed in the introduction, Beise and M\"uller showed that there exist functions in the Hardy space $H^p$, $1< p<\infty$, with universal Taylor series on closed sets $E\subset\mathbb{T}$ of measure $0$. They observed that the hypothesis that $E$ has zero measure is sharp, since by the Carleson-Hunt theorem, the Taylor series of a function in $H^p$ converges to its boundary values a.e.~on $\mathbb{T}$.
In our setting of optimal polynomial approximants, this assumption also cannot be dispensed with: if $E\subset\mathbb{T}$ has positive measure, then $\mathcal{U}_{E}=\emptyset$. To see this,  assume for the sake of contradiction that $\mathcal{U}_{E}\neq\emptyset$. By the remark in Section \ref{sec:existence-opa}, we would be able to find a cyclic function $f \in \mathcal{U}_{E}$.
 Let $\ell \in \mathbb{C}$ be an arbitrary complex number and $(q_k)$ be a subsequence of o.p.a.~of $1/f$ which converges to $\ell$ uniformly on $E$.
 Since $f$ is cyclic, $\|q_kf-1\|_{L^2(\mathbb{T})}=\|q_kf-1\|_{H^2} \to 0$
 (where we continue to denote by $f$ its boundary values on $\mathbb{T}$). This implies that some subsequence of $(q_k)$ tends to $1/f$  a.e.~on $\mathbb{T}$. Thus $1/f=\ell$ a.e. on $E$, which yields a contradiction because $E$ has positive measure and $\ell$ is arbitrary.
\item In \cite{beise-muller}, Beise and M\"uller showed that there exists a function in the Bergman space $A^2$ that has a universal Taylor series on a subset of the circle of positive measure and containing no closed subset of positive measure and finite entropy. \color{black}  It would be interesting to extend our results for optimal polynomial approximants to this setting.  However, our proof of the simultaneous zero-free approximation does not carry over.
\item It might be interesting to examine whether there exist functions with universal o.p.a.~on sets that are not necessarily contained in the unit circle.  Results in this direction for the Taylor case have been obtained in \cite{CJM}.
\end{enumerate}

\bigskip

\noindent\textbf{Acknowledgements.} Myrto Manolaki thanks the Department of Mathematics and Statistics at the University of South Florida for support during work on this project. Daniel Seco acknowledges financial support from the Spanish Ministry of Economy and Competitiveness through the ``Severo Ochoa Programme for Centers of Excellence in R\&D'' (SEV-2015-0554) and through the grant MTM2016-77710-P.


\begin{thebibliography}{ll}

\bibitem{andersson} J.~Andersson, \textit{Mergelyan's approximation theorem with nonvanishing polynomials and universality of zeta-functions}\/, J. Approx. Theory 167 (2013), 201--210.

\bibitem{andersson-gauthier} J.~Andersson, P.~M.~Gauthier, \textit{Mergelyan's theorem with polynomials non-vanishing on unions of sets}\/, Complex Var. Elliptic Equ. 59 (2014), no. 1, 99--109.

\bibitem{beise-muller} H.-P.~Beise, J.~M\"uller, \textit{Generic boundary behaviour of Taylor series in Hardy and Bergman spaces}\/, Math. Z. 284 (2016), no. 3--4, 1185--1197.

\bibitem{opa0}  C.~B\'en\'eteau, A.~Condori, C.~Liaw, D.~Seco, A.~Sola, \textit{Cyclicity in Dirichlet-type Spaces and Extremal Polynomials}\/,  J. Anal. Math. 126 (2015), 259--286.

\bibitem{opa4} C.~B\'en\'eteau, M.~Fleeman, D.~Khavinson, D.~Seco, A.~Sola, \textit{Remarks on Inner Functions and Optimal Approximants}\/, Canad. Math. Bull. 61 (2018), 704--716.

\bibitem{opa3} C.~B\'en\'eteau, D.~Khavinson, C.~Liaw, D.~Seco, B.~Simanek, \textit{Zeros of optimal polynomial approximants: Jacobi matrices and Jentzsch-type theorems}\/, to appear in Rev. Mat. Iberoam., \href{http://arxiv.org/pdf/1606.08615}{arXiv:1606.08615}.

\bibitem{opa2} C.~B\'en\'eteau, D.~Khavinson, C.~Liaw, D.~Seco, A.~Sola, \textit{Orthogonal Polynomials, Reproducing Kernels, and Zeros of Optimal Approximants}\/,  J. London Math. Soc. 94 (2016), no. 3, 726--746.

\bibitem{opa1} C.~B\'en\'eteau, M.~Manolaki, D.~Seco, \textit{Boundary Behavior of Optimal Polynomial Approximants}\/, preprint, \href{http://arxiv.org/pdf/1901.00694}{arXiv:1901.00694}.


\bibitem{brown-shields} L.~ Brown, A.~L.~ Shields, \textit{Cyclic vectors in the Dirichlet space}\/, Trans. Amer. Math.
Soc. 285 (1984), 269--304.

\bibitem{carleson} L.~Carleson, \textit{Representations of continuous functions}\/, Math. Z. 66 (1957), no. 1, 447--451.

\bibitem{CJM} G.~Costakis, A.~Jung, J.~M\"uller, \textit{Generic Behavior of Classes of Taylor Series Outside the Unit Disk}\/, Constr. Approx., April 2018.

\bibitem{danielyan}  A.~Danielyan, \textit{On the zero-free polynomial approximation problem}\/, J. Approx. Theory 205 (2016), 60--63.

\bibitem{duren} P.~L.~Duren, \textit{Theory of $H^p$ spaces}\/, Academic Press, New York, 1970.

\bibitem{duren-schuster} P.~L.~Duren and A.~Schuster, \textit{Bergman spaces}, AMS, Providence, RI, 2004.

\bibitem{EFKMR} O.~El-Fallah, K.~Kellay, J.~Mashreghi, and T.~Ransford, \textit{A primer on the Dirichlet space}, Cambridge University Press, 2014.

\bibitem{gardiner14} S.~J.~Gardiner, \textit{Universal Taylor series, conformal mappings and boundary behaviour}, Ann. Inst. Fourier 64 (2014), no. 1, 327--339.

\bibitem{gardiner18} S.~J.~Gardiner, \textit{Universality and Potential Theory}, New Trends in Approximation Theory, Fields Institute Communications 81 (2018), 247--264, Springer, New York, NY.

\bibitem{gardiner-khavinson} S.~J.~Gardiner, D.~Khavinson, \textit{Boundary behaviour of universal Taylor series}\/, C. R. Math. Acad. Sci. Paris 352 (2014), no. 2, 99--103.

\bibitem{gardiner-manolaki3} S.~J.~Gardiner, M.~Manolaki, \textit{Boundary behaviour of Dirichlet series with applications to universal series}, Bull. London Math. Soc. 48 (2016), no. 5, 735--744.

\bibitem{gardiner-manolaki2} S.~J.~Gardiner, M.~Manolaki, \textit{A convergence theorem for harmonic measures with applications to Taylor series}, Proc. Amer. Math. Soc. 144 (2016), no. 3, 1109--1117.

\bibitem{garnett} J.~B.~Garnett, \textit{Bounded analytic functions}, Pure and Applied Mathematics, 96, Academic Press, New York-London, 1981.

\bibitem{garnett-marshall} J.~B.~Garnett, D.~E.~Marshall, \textit{Harmonic  Measure}, New Mathematical Monographs 2, Cambridge University Press, 2005.

\bibitem{gauthier-knese} P.~M.~Gauthier, G.~Knese, \textit{Zero-free polynomial approximation on a chain of Jordan domains}\/, Ann. Sci. Math. Qu\'ebec 36 (2012), no. 1, 107--112.

\bibitem{muller2} J.~M\"uller, \textit{Spurious boundary limit functions}, to appear in Anal. Math. Phys.

\bibitem{nestoridis} V.~Nestoridis, \textit{Universal Taylor series}\/, Ann. Inst. Fourier (Grenoble) 46
(1996), 1293--1306.

\bibitem{rudin} W.~Rudin, \textit{Boundary values of continuous analytic functions}\/,
Proc. Amer. Math. Soc. 7 (1956), 808--811.

\end{thebibliography}
\end{document}